\documentclass[reqno, 12pt, reqno]{amsart}

\usepackage{amsfonts, amsthm, amsmath, amssymb}

\usepackage{hyperref}
\usepackage{a4wide}

\RequirePackage{mathrsfs} \let\mathcal\mathscr

\DeclareRobustCommand{\SkipTocEntry}[5]{} 

\numberwithin{equation}{section}

\newtheorem{theorem}{Theorem}[section]
\newtheorem{lemma}[theorem]{Lemma}

\theoremstyle{definition}

\newtheorem*{rem*}{Remark}

\renewcommand{\rho}{\varrho}

\newcommand{\ZZ}{\mathbb{Z}}
\newcommand{\NN}{\mathbb{N}}

\newcommand{\CC}{\mathbb{C}}

\renewcommand{\leq}{\leqslant}
\renewcommand{\geq}{\geqslant}

\newcommand{\Mod}[1]{\;(\operatorname{mod}\,#1)}

\newcommand{\id}{\rm{id}}

\DeclareMathOperator{\lcm}{lcm}
\newcommand{\eps}{\varepsilon}

\begin{document}

\title[On the factorisation theorem for polynomial orbits]
{A consequence of the factorisation theorem for polynomial orbits on nilmanifolds 
\\ -- \\ Corrigendum to {\em Acta Arith.} {\bf 154} (2012), 235--306}
\author{Lilian Matthiesen}
\address{}
\email{matthiesen@math.uni-hannover.de}

\begin{abstract}
We discuss a consequence of Green and Tao's factorisation theorem for polynomial 
orbits on nilmanifolds, adjusted to the requirements of certain arithmetic 
applications.
More precisely, we prove a generalisation of
Theorem 16.4, {\em Acta Arith.} {\bf 154} (2012), 235--306, by slightly 
rearranging its proof.
The thus achieved strengthening of the result removes an oversight in the 
above-cited paper which resulted from the previously too weak conclusion.
Since this type of result proved essential for further applications, we 
take the opportunity to discuss it in more detail.
\end{abstract}

\maketitle

\section{Introduction}

The quantitative factorisation theorem for polynomial orbits on nilmanifolds that 
was proved in Green--Tao \cite{GT-polyorbits} plays a fundamental role in 
applications of Green and Tao's `nilpotent Hardy--Littlewood method', the 
machinery behind their celebrated work \cite{GT-linearprimes}.
The theorem allows one to factorise an arbitrary polynomial sequence $g$ on a 
nilmanifold as a product $\eps g' \gamma$ of three polynomial sequences, where 
$\eps$ is slowly varying, $g'$ is highly equidistributed and $\gamma$ is periodic.
It is usually the equidistribution properties of the sequence $g'$ one seeks to 
exploit.
In order to access these one splits the range of $n$ into subprogressions on which 
$\eps$ is almost constant and on which $\gamma$ is constant.
For this approach to work it is crucial that $g'$ is still equidistributed when 
restricted to these subprogressions.

The nilpotent Hardy--Littlewood method can be used to study correlations of 
arithmetic functions; in Green and Tao's case the function studied is the von 
Mangoldt function. 
It is usually desirable that the relevant arithmetic function, $f$ say, has the 
same average value on each of the new subprogressions as on the original 
(possibly $W$-tricked) range.
We will refer to this property as the major arc property.

The purpose of this paper is to discuss a consequence of Green and Tao's 
factorisation theorem that is suitable for specific arithmetic applications such 
as questions involving a function $f$ that counts the number of representations of 
an integer by a form or a polynomial, c.f. \cite{lmr} and \cite{bm}.
In particular, we are interested in functions $f$ whose average value in 
arithmetic progressions is determined by a product of local densities.
The problem one faces in such applications is the following:

The product of local densities will typically remain constant when the $p$-adic 
densities of elements within two arithmetic progressions is the same.
In other words, if $\{n \equiv r_1 \Mod{q_1}\}$ and $\{n \equiv r_2 \Mod{q_2}\}$ 
are two progressions such that $p|q_1 \Leftrightarrow p|q_2$ and such that
$r_1 \equiv r_2 \Mod{\lcm(q_1,q_2)}$, then the average value of $f$ restricted to 
either progression is expected to be the same.
This information allows us to employ a $W$-trick to obtain a major arc property as 
follows. 
Let $W$ be a product of small prime powers.
Then the average value of $f$ is preserved when passing from a 
progression of the form $\{W n + A\}$, $A \not\equiv 0 \Mod{p^{v_p(W)}}$ for 
any $p|W$, to a subprogression of the form $\{W(qn+r)+A\}$ where $q$ is entirely 
composed of primes dividing $W$.

When applying the factorisation theorem it is necessary to ensure that the 
periodic sequence $\gamma$ has a period that belongs to the set of integers $q$ 
from above.
The only way to do so in the setting of the original factorisation theorem is to 
ensure that the upper bound, $Q$, for the period of $\gamma$ does not exceed the 
largest prime factor dividing $W$.
This however produces at the same time an upper bound of the shape $Q^{A}$, for 
some fixed $A\geq 1$ that we are free to choose, on the common difference of 
subprogressions on which $g'$ is still equidistributed.

In an arithmetic application with polynomial structure one is however naturally 
lead to situations where one would like to know that for instance
$n \mapsto g'(aW^{\ell} n +b)$ for some fixed ${\ell} \in \NN$ is still 
equidistributed to some extent.
Since $A$ is fixed whereas $W=W(N)$ usually grows with $N$, this is difficult to 
achieve with the original factorisation theorem.
More precisely, while $q=aW^{\ell}$ would be a valid choice for the major 
arc condition provided $q=a$ is, it is impossible to simultaneously guarantee 
that $W^{\ell} < Q^{A}$ holds and that $Q \leq P^+(W)$ is bounded above by the 
largest prime factor of $W$.

To surpass this difficulty, a modified factorisation theorem was established in 
\cite[\S16]{lmr}.
While the chosen strategy of proof is sufficiently powerful, it has unfortunately 
been overlooked that in the form the result is stated not all requirements are 
met; specifically, \cite[Theorem 16.4]{lmr} does not guarantee that $g'$ is 
equidistributed on the subprogressions on which $\gamma$ is constant.
By going through the proof slightly more carefully, we obtain in Theorem 
\ref{l:factorisation} below a generalisation of the result which resolves the 
issue.

\section{A new factorisation lemma for polynomial nilsequences}
To start with, we recall the statement of the factorisation theorem from
Green and Tao \cite[Thm 1.19]{GT-polyorbits}:

\begin{theorem}\label{t:GT}
Let $m, d > 0$, and let $M_0 , N > 1$ and $A > 0$ be real numbers. 
Suppose that $G/\Gamma$ is an $m$-dimensional nilmanifold together with a
filtration $G_{\bullet}$ of degree $d$. 
Suppose that $X$ is an $M_0$-rational Mal'cev basis $\mathcal{X}$ adapted to
$G_{\bullet}$ and that $g \in \mathrm{poly}(\ZZ, G_{\bullet})$. 
Then there is an integer $M$ with $M_0 \leq M \ll M_0^{O_{A,m,d}(1)}$,
a rational subgroup $G' \subseteq G$, a Mal'cev basis $\mathcal{X}'$ for 
$G'/\Gamma'$ in which each element is an $M$-rational combination of the elements 
of $\mathcal{X}$, and a decomposition $g = \eps g' \gamma$ into
polynomial sequences $\eps, g', \gamma \in \mathrm{poly}(\ZZ, G_{\bullet})$ with 
the following properties:
\begin{enumerate}
 \item $\eps : \ZZ \to G$ is $(M, N)$-smooth;
 \item $g': \ZZ \to G'$ takes values in $G'$, and the finite sequence 
$(g′(n)\Gamma')_{n \leq N}$ is totally $1/M^A$-equidistributed in $G'/\Gamma'$, 
using the metric $d_{\mathcal{X}'}$ on $G'/\Gamma'$;
 \item $\gamma : \ZZ \to G$ is $M$-rational, and $(\gamma(n)\Gamma)_{n \in \ZZ}$ 
is periodic with period at most $M$.
\end{enumerate} 
\end{theorem}

We will employ this result in an iterative process.
To guarantee the termination of this process, we will ensure that in each 
of our applications of the above result the rational subgroup $G'$ will be of 
strictly lower dimension than that of the ambient group $G$.

\begin{lemma}\label{l:G=G'}
Under the hypotheses of Theorem \ref{t:GT} let 
$g \in \mathrm{poly}(\ZZ, G_{\bullet})$ and suppose that 
$g(n)= \eps(n) g'(n) \gamma(n)$ is a factorisation which satisfies the 
conditions (1)--(3).
Then there is a positive constant $C$ only depending on $m$ and $d$ such that 
whenever $A$ is sufficiently large and $G'=G$, then 
$g$ is totally $M^{-A/2C}$-equidistributed.
\end{lemma}
\begin{proof}
 Let $C \geq 1$ to be determined in the course of the proof and let
$P\subseteq\{1, \dots, N\}$ be a progression of length at least $M^{-A/2C}N$.
Since $\gamma$ is periodic with period bounded above by $M$, we may split $P$ 
into at most $M$ subprogressions, each of length at least $M^{-(A/(2C) + 1)}N$, 
on which $\gamma$ is constant.
Next, we split each of these subprogressions into pieces of diameter between
$M^{-((A/(2C) + 1)}N$ and $2M^{-(A/(2C) + 1)}N$ and let $\mathcal{P}$ denote the 
collection of all resulting bounded diameter pieces.
For each progression $Q \in \mathcal{P}$, let $s_{Q}$ denote its smallest 
element.
If $F: G/\Gamma \to \CC$ is a Lipschitz function, then the right-invariance of 
the metric $d_{\mathcal{X}}$ (c.f. \cite[Definition 2.2]{GT-polyorbits}) implies 
for any $n,n'$ that belong to the same element $Q$ of $\mathcal{P}$ that:
\begin{align*}
|F(\eps(n) g'(n) \gamma(n) \Gamma)
- F(\eps(n') g'(n) \gamma(n) \Gamma)|
&\leq \|F\|_{\mathrm{Lip}}~ 
d_{\mathcal{X}}(\eps(n) g'(n) \gamma(n),\eps(n') g'(n) \gamma(n)) \\
&= \|F\|_{\mathrm{Lip}}~ d_{\mathcal{X}}(\eps(n),\eps(n') ) \\
&\leq \|F\|_{\mathrm{Lip}}~ |n-n'|M/N \\
&\leq 2\|F\|_{\mathrm{Lip}}~ M^{-A/(2C)}.
\end{align*}
This estimate allows one to fix for any $Q \in \mathcal{P}$ the contribution of 
$\eps$:
$$
\sum_{n\in Q} F(g(n)\Gamma) 
= \sum_{n\in Q} F(\eps(s_{Q}) g'(n) \gamma(n) \Gamma)
 + O(\#Q \|F\|_{\mathrm{Lip}} M^{-A/(2C)}).
$$
Let $H_Q: G/\Gamma \to \CC$ denote the map $H_Q(h) := F(\eps(s_{Q}) h \Gamma)$ 
and observe that the approximate left-invariance of $d_{\mathcal{X}}$ (c.f. 
\cite[Lemma A.5]{GT-polyorbits}) implies that 
$\|H_Q\|_{\mathrm{Lip}} \leq M_0^{O(1)} \|F\|_{\mathrm{Lip}}$. 
Furthermore we have $\int_{G/\Gamma} F = \int_{G/\Gamma} H_Q$. 
The fact that $(g'(n)\Gamma)_{n\leq N}$ is totally 
$M^{-A}$-equidistributed in $G/\Gamma$ allows us to deduce a similar property 
for each of the sequences $(g'(n) \gamma(m)\Gamma)_{n\leq N}$ for fixed $m$.
It follows from \cite[Proposition 14.3]{lmr}, which is a consequence of
\cite[Theorem 2.9]{GT-polyorbits}, that there is a constant 
$C' = BB' > 1$, only depending on $m$ and $d$, such that 
$(g'(n)\gamma(m)\Gamma)_{n\leq N}$ is totally $M^{-A/C}$-equidistributed.
We set $C=C'$.
Applying the above to any progression $Q \in \mathcal{P}$, we obtain
\begin{align*}
\sum_{n\in Q} F(\eps(s_{Q}) g'(n) \gamma(n) \Gamma)
&= \sum_{n\in Q} F(\eps(s_{Q}) g'(n) \gamma(n_Q) \Gamma)\\
&= \left( \int_{G/\Gamma} F + 
   O \left(M_0^{O(1)}M^{-A/C} \|F\|_{\mathrm{Lip}}\right) \right)\#Q,
\end{align*}
and, hence,
$$
\sum_{n\in N} F(g(n)\Gamma)
= N \left( \int_{G/\Gamma} F + \|F\|_{\mathrm{Lip}} 
 O\left(M^{-A/(2C)}+M_0^{O(1)}M^{-A/C}\right) 
\right).
$$
This completes the proof.
\end{proof}

\begin{theorem}[Factorisation lemma] \label{l:factorisation}
Let $N$ and $T=T(N)$ be positive integer parameters that satisfy
$N^{1-\eps} \ll_{\eps} T \ll N$ and let $k:\NN \to \NN$ be a slowly 
growing function.
Let $m$, $d$, $B$, $E$ and $Q_0 \geq 1$ be positive integers.
Suppose that $G/\Gamma$ is an $m$-dimensional nilmanifold together with a
filtration $G_{\bullet}$ of degree $d$. 
Suppose that $\mathcal X$ is a $Q_0$-rational Mal'cev basis
adapted to $G_{\bullet}$, and that $g \in \mathrm{poly}(\ZZ,G_{\bullet})$.
Suppose further that $Q_0 \leq \log k(N)$. 
Let $R=R(N)$ be a parameter that satisfies $R \geq Q_0$ and 
$R(N)^t \ll_t N$ for all $t>0$.
Then there is an integer $Q$ with 
$Q_0 \leq Q \ll Q_0^{O_{B,m,d}(1)}$
and a partition of $\{1, \dots, T\}$ into at most 
$R^{O_{m,d,B,E}(1)}$ disjoint subprogressions $P$, 
each of $k(N)$-smooth common difference 
$q(P) \ll R^{O_{m,d,B,E}(1)}$
and each of length $T/q(P) + O(1)$, 
such that the restriction $(g(n))_{n \in P}$ of $g$ to any of the 
progressions $P$ can be factorised as follows.

There is a rational subgroup $G' \leq G$, depending on $P$, and 
a Mal'cev basis $\mathcal X'$ for $G'\Gamma/\Gamma$ such that 
every element of $\mathcal X'$ is a $Q$-rational combination of elements
from $\mathcal X$ (that is, each coefficient is rational of height
bounded by $Q$).
Suppose $P= \{qn+r: 1 \leq n \leq T/q + O(1)\}$, where $q=q(P)$, then we have a 
factorisation
$$g(q n + r) = \eps_P(n) g'_P(n) \gamma_P(n)~,$$
where $\eps_P, g'_P, \gamma_P$ are polynomial
sequences from $\mathrm{poly}(\ZZ, G_{\bullet})$ with the properties
\begin{enumerate}
 \item $\eps_P: \ZZ \to G$ is $(Q,T/q)$-smooth\footnote{The notion of 
 smoothness was defined in \cite[Def.\ 1.18]{GT-polyorbits}. 
 A sequence $(\eps (n))_{n\in \ZZ}$ is said to be \emph{$(M,N)$-smooth} if 
 both $d_{\mathcal{X}}(\eps(n),\id_G) \leq N$ and 
 $d_{\mathcal{X}}(\eps(n),\eps(n-1)) \leq M/N$ hold 
 for all $1 \leq n \leq N$.};
 \item $\gamma_P: \ZZ \to G$ arises as the product of at most $m$ 
 $Q$-rational\footnote{A sequence $\gamma: \ZZ \to G$ is said to be 
 \emph{$Q$-rational} if for each $n$ there is $0<r_n\leq Q$ such that 
$(\gamma(n))^{r_n} \in \Gamma$. 
 See \cite[Def.\ 1.17]{GT-polyorbits}.} polynomial sequences and the sequence 
$(\gamma_P(n)\Gamma)_{n \in \ZZ}$ is periodic with a 
 $k(N)$-smooth period $q_{\gamma_P} \leq Q$;
 \item
 $g'_P: \ZZ \to G'$ takes values in $G'$ and for each $k(N)$-smooth number 
$\tilde q < (q q_{\gamma_P}R)^{E}$ the finite sequence 
$(g'_P(\tilde q n)\Gamma')_{n \leq T/(q \tilde q)}$ is totally
$Q^{-B}$-equidistributed in $G'\Gamma/\Gamma$.
\end{enumerate}
\end{theorem}

The proof of the factorisation lemma makes use of the fact that a polynomial 
sequence that fails to be totally equidistributed also fails to be 
equidistributed when allowing polynomial changes in the equidistribution 
parameter.
This is made precise in \cite[Lemma 6.2]{lmm}, which we restate here for 
simplicity:
\begin{lemma}\label{l:equi/totally-equi}
 Let $N$ and $A$ be positive integers and let $\delta: \NN \to [0,1]$ be a 
 function that satisfies $\delta(x)^{-t} \ll_t x$ for all $t>0$.
 Suppose that $G$ has $\frac{1}{\delta(N)}$-rational Mal'cev basis adapted to the 
 filtration $G_{\bullet}$. 
 Suppose that $g \in \mathrm{poly}(\ZZ,G_{\bullet})$ is a polynomial sequence 
 such that $(g(n)\Gamma)_{n\leq N}$ is $\delta(N)^A$-equidistributed.
 Then there is  $1\leq B \ll_{d_G,m_G} 1$
 such that $(g(n)\Gamma)_{n\leq N}$ is totally 
 $\delta(N)^{A/B}$-equidistributed, provided $A/B>1$ and provided $N$ is 
 sufficiently large.
\end{lemma}

\begin{proof}[Proof of Lemma \ref{l:factorisation} assuming Lemma 
\ref{l:equi/totally-equi}]
We may suppose that $g$ does not satisfy (3) with $Q$ replaced by $Q_0$. 
That is, there is some $k(N)$-smooth integer $q_1 \leq R^{E}$ such 
that $(g(q_1 n)\Gamma)_{n \leq T/q_1}$ fails to be totally
$Q_0^{-B}$-equidistributed.
By Lemma \ref{l:equi/totally-equi}, this sequence also fails to be 
$Q_0^{-BC}$-equidistributed for some $C = O_{m,d}(1)$.
Writing $z_1:=(q_1)^d$, we deduce from \cite[Lemma 16.3]{lmr} 
that each of the sequences $(g(z_1 n + r_1)\Gamma)_{n \leq T/z_1}$ with
$0 \leq r_1 < z_1$ fails to be $Q_0^{-BCC'}$-equidistributed in $G/\Gamma$ for 
some $C'=O_{m,d}(1)$.
Now, we run through all $0\leq r_1 < z_1$ in turn.

Applying Theorem \ref{t:GT} and Lemma \ref{l:G=G'} to any of these sequences 
yields some 
$Q_0 \leq Q_1 \ll Q_0^{O(B,m,d)}$, a $Q_1$-rational subgroup $G_1 < G$ of 
dimension strictly smaller than that of $G$, and a factorisation 
$$g(z_1 n + r_1) = \eps_{r_1}(n) g'_{r_1}(n)\gamma_{r_1}(n),$$
where the finite sequence $(g'_{r_1}(n) \Gamma_1)_{n\leq T/z_1}$ is totally
$Q_1^{-B}$-equidistributed in  $$G_1/\Gamma_1:=G_1/(\Gamma \cap G_1),$$
where $(\eps_{r_1}(n)\Gamma)_{n\in\ZZ}$ is $(Q_1, T/z_1)$-smooth, and
where $(\gamma_{r_1}(n)\Gamma)_{n\in\ZZ}$ is periodic with period at most $Q_1$.

If $g'_{r_1}$ is totally $Q_1^{-B}$-equidistributed on every subprogression 
$\{n \equiv 0 \Mod{q_2} \}$ of $k(N)$-smooth common difference 
$q_2< (z_1 Q_1 R)^{E}$, then we stop 
(and turn to the next choice of $r_1$). 
Otherwise, invoking Lemma \ref{l:equi/totally-equi} and \cite[Lemma 16.3]{lmr} 
again, there are positive integers $C, C' = O_{d,m}(1)$ and
a $k(N)$-smooth integer $q_2$ as above such that, with $z_2:=q_2^d$, the finite 
sequence $(g_{r_1,r_2}(n))_{n \leq T/(z_1z_2)}$ defined by 
$g_{r_1,r_2}(n) := g'_{r_1}(z_2 n + r_2)$ fails to be
$Q_1^{-BCC'}$-equidistributed for every $0 \leq r_2 < z_2$. 
We proceed as before.

This process yields a tree of operations which has height at most 
$m = \dim G$, since each time the factorisation theorem is applied a new
sequence $g'_{r_1, \dots, r_i}$ is found that takes values in some
strictly lower dimensional submanifold 
$G_i= G_i(r_1, \dots, r_i)$ of $G_{i-1}(r_1, \dots, r_{i-1})$. 
Thus, we can apply the factorisation theorem at most $m$ times in a row
before the manifold involved has dimension $0$. 

The tree we run through starts with $g$, which has $z_1$
neighbours $g_{r_1}$, one for each $0\leq r_1 < z_1$. 
For each $r_1$, the vertex $g_{r_1}$ has $z_2=z_2(r_1)$ neighbours 
$g_{r_1,r_2}$, one for each $0 \leq r_2 < z_2(r_1)$, etc.
As a result, we obtain a decomposition of the range 
$\{1, \dots, T\}$ into at most
$R^{O_{m,d,B,E}(1)}$ 
subprogressions of the form 
\begin{align*}
P
& = \{z_1(z_2(z_3( \dots (z_t n + r_t)
     \dots) + r_3) + r_2) + r_1 
    ~:~ n \leq T/(z_1z_2 \dots z_t) \} \\
& = \{ z_1z_2 \dots z_t n + r 
    ~:~ n \leq T/(z_1z_2 \dots z_t)\} ~,
\end{align*}
for $t \leq m$, some $r$, and where each $z_i$ depends on $r_1, \dots,r_{i-1}$.
The common difference of such a progression $P$ is
$k(N)$-smooth and bounded by 
$R^{O_{m,d,B,E}(1)}$. 
The iteration process furthermore yields a factorisation of 
$g_{r_1, \dots, r_t}$, which is the restriction of $g$ to $P$:
$$ g_{r_1,\dots,r_t} (m)
=g(z_1z_2 \dots z_t m + r) 
= \tilde \eps_{r_1,\dots,r_t}(m)
  g'_{t}(m)
  \tilde \gamma_{r_1,\dots,r_t}(m)~,
$$ 
where
$$\tilde \eps_{r_1,\dots,r_t}(m)
= \eps_{r_1}( z_2 \dots z_t m + \tilde r_2) \dots
  \eps_{r_1,\dots,r_{t-1}}(z_t m + \tilde r_t)
  \eps_{r_1,\dots,r_t}(m)
$$
for certain integers $\tilde r_2, \tilde r_3, \dots, \tilde r_t$,
and
$$ \tilde \gamma_{r_1,\dots,r_t}(m)
  =
  \gamma_{r_1,\dots,r_t}(m)
  \gamma_{r_1,\dots,r_{t-1}} (z_t m + \tilde r_t) \dots
  \gamma_{r_1}(z_2 \dots z_t m + \tilde r_2)~.
$$

The factor $\tilde \eps_{r_1,\dots,r_t}(m)$
is a $(Q_0^{O_{B,d,m}(1)},T/ (z_1 \dots z_t))$-smooth sequence.
This follows from the triangle inequality, the right-invariance and the 
approximate left-invariance of $d_{\mathcal{X}}$; we refer to the discussion 
following Definition 16.1 in \cite{lmr} for details and to 
\cite[App.~A]{GT-polyorbits} for the properties of $d_{\mathcal{X}}$.

Since each $\gamma_{r_1,\dots,r_{i}}$ with $i \leq t$ is periodic with 
period at most $Q_0^{O_{m,d,B}(1)}$ and since $t \leq m$, we deduce that
$\tilde \gamma_{r_1,\dots,r_t}$ is periodic with period at most 
$Q_0^{O_{m,d,B}(1)}$.
The bound $Q_0 \leq \log k(N)$ implies that this period is $k(N)$-smooth provided 
$N$ is sufficiently large.

Finally, $g'_t$ satisfies property (3) by construction.
\end{proof}

\end{document}